\documentclass{amsart}

\newtheorem{thm}{Theorem}
\newtheorem{lem}{Lemma}
\newtheorem{prop}{Proposition}

\usepackage{amsthm, amssymb}
\usepackage{mathrsfs}

\def\E{\mathbb{E}}

\def\S{\mathbf{S}}
\def\half{\frac{1}{2}}
\def\d{\mathrm{d}}
\def\T{\mathbb{T}}

\title{$6$-th Norm of a Steinhaus Chaos}

\author{Kamalakshya Mahatab}
\address{Kamalakshya Mahatab, Department of Mathematical Sciences, NTNU Trondheim, Norway}
\email{accessing.infinity@gmail.com, kamalakshya.mahatab@ntnu.no}

\thanks{This work was carried out during the tenure of an ERCIM ``Alain Bensoussan'' Fellowship 
of the author. The author is also supported by Grant 227768 of the Research Council of Norway}

\begin{document}

\begin{abstract}
We prove that for the Steinhaus Random Variable $z(n)$,
$$\E\left(\left|\sum_{n\in E_{N, m}}z(n)\right|^6\right)\asymp
|E_{N, m}|^3 \text{ for } m\ll(\log\log N)^{\frac{1}{3}},$$
where 
\[E_{N, m}:=\{1\leq n:\Omega(n)=m\}\]
 and $\Omega(n)$ denotes the number of prime factors of $N$.
\end{abstract}

\maketitle

\section{Introduction}\label{sec:intro}
Let $(z(p))_{p \ \text{prime}}$ be the Steinhaus random variable, equidistributed on the unit circle $\mathbb{T}:=\{s\in \mathbb{C}: |s|=1\}$.
This function can be extended to all natural numbers by defining it completely multiplicatively. We define
\begin{align*}
 S_N(z)&:=\sum_{1\leq n\leq N}z(n) \text{ and }\\
 S_{N, m}(z)&:=\sum_{n\in E_{N, m}}z(n), \text{ where }\\
 E_{N, m}&:=\{1\leq n:\Omega(n)=m\}.
\end{align*}
Expectations of such Steinhaus chaoses, $S_N$ and $S_{N, m}$, received attentions from several mathematicians in recent years due to its connections to number theory and
harmonic analysis \cite{MR3589675}. The expectations of $S_N$ and $S_{N, m}$ are defined as follows
\[\E(|S_N|^q):=\int_{\T^\infty}|S_N(z)|^q\d \mu_\infty(z) \text{ and }\E(|S_{N, m}|^q):=\int_{\T^\infty}|S_{N, m}(z)|^q\d \mu_\infty(z),\]
where $z$ denotes the coordinate tuple $(z(p_j))$, $p_j$ being the $j$th prime, and $\mu_\infty$ is the normalized Haar measure on infinite dimensional torus $\T^\infty$. By Bohr correspondence 
we have the following well known identity for all $q>0$ (see \cite[Section 3]{MR2506825})
\[\E(|S_N|^q)=\lim_{T\rightarrow\infty}\frac{1}{T}\int_0^T\left|\sum_{n=1}^{N}n^{-it}\right|^q\d t \text{ and }\E(|S_{N, m}|)=\lim_{T\rightarrow\infty}\frac{1}{T}\int_0^T\left|\sum_{n\in E_{N, m}}n^{-it}\right|^q\d t.\]
In \cite{MR2640082}, Helson observed that if 
\begin{equation}\label{eq:helson}
 \E(|S_N|)=o(\sqrt N)
\end{equation}
then Nehari's theorem on boundedness of Hankel forms does not extend to $\mathbb{T}^\infty$. 
While Nehari's theorem has shown to fail on $\mathbb{T}^\infty$ by means of another Dirichlet polynomial\cite{MR3589675}, 
the question of whether (\ref{eq:helson}) holds remained open and
was proved only recently by Harper \cite{harper}. In an interesting approach to obtain a lower bound for $\E(|S_N|)$, Bondarenko and Seip \cite{MR3430378} showed that
\[\|S_{N, m}\|_2\asymp \|S_{N, m}\|_4 \ \text{ for } m<\beta\log\log N,\]
where $\beta<\half$ and $\|S_{N, m}\|_q:=\E(|S_N|^q)^{1/q}$. This implies
\[\E(|S_N|)\gg\frac{\sqrt N}{(\log\log N)^{0.05616}} \text{ and } |S_{N, m}\|_q\gg_q \frac{\sqrt N}{(\log\log N)^{0.07672}} \text{ for }q>0.\]
In this article, we will investigate the following question:
~\\

\textsl{Is there a constant $c(k)$, for each $k$, such that $\|S_{N, m}\|_{2k}\asymp \|S_{N, m}\|_2$} when 

$m<\beta\log\log N$ and $\beta<c(k)$?

~\\
We conjecture that such a constant exists for each $k$, but proving this statement seems difficult. Instead we will show the following:
\begin{thm}\label{thm:main}
 For $m\ll(\log\log N)^{\frac{1}{3}}$,
 \[\|S_{N, m}\|_6\asymp\|S_{N, m}\|_4 \ \text{ as } N\rightarrow\infty.\]
\end{thm}
It is difficult to generalize our proof for $6$-th moment to other even moments as the computation becomes extremely complicated.
We may note that the constant in the above asymptotic is independent of $m$ and $N$.
Certain computations in the proof of the this theorem indicate us to conjecture that $c(3)=\frac{1}{4}$. 

We may compare the above result to a result of Hough \cite{MR2770059} on Rademacher 
random variable. Let $f$ denotes the Rademacher random variable defined on primes, and takes the values $\pm1$ with probability $\half$ each. Further, extend $f$ to
all natural numbers by defining it as a completely multiplicative function. Let
\[S_{N, m, f}:=\sum_{n\in E_{N, m}}f(n).\]
Then ( see Proposition~10 \cite{MR2770059})
 \[\E(|S_{N, m, f}|^{2k})\asymp|E_{N, m}|^k \text{ for } m=o(\log\log\log N).\]
Our theorem gives a better range for $m$ when $k=3$ in case of the Steinhaus random variable.
However, we can not compute higher moments of $S_{N, m}$ using Hough's technique, nor our technique can be applied to compute the $6$-th moment of $S_{N, m, f}$.

It is known that the distribution of $\frac{S_{N, m, f}}{\|S_{N, m}\|_2}$
is standard normal when $m$ is sufficiently small compare to $N$ (see \cite{MR2770059},\cite{Harper2}).
We may expect similar results for Steinhaus random variable, and conjecture that when $m=o(\log\log N)$, the distribution of $\frac{S_{N, m}}{\|S_{N, m}\|_2}$ should converge to a complex normal distribution and 
\[\lim_{N\rightarrow\infty}\frac{\|S_{N, m}\|_{2k}}{\|S_{N, m}\|_2}=\frac{(2k)!}{2^k(k!)}.\]

We will simplify $\E(|S_{N, m}|^{6})$ in Section~\ref{sec:identity} and prove some preparatory lemmas in Section~\ref{sec:lemm} and \ref{sec:S}. 
In Section~\ref{sec:proof} we will give a proof of Theorem~\ref{thm:main}.

\section{An Expression for $6^\text{th}$ Norm of $S_{N, m}$}\label{sec:identity}
\begin{prop}
We have the following expressions for $\|S_{N, m}\|_6^6$
 \begin{align}\label{eq:identity_6thmoment}
 &\|S_{N, m}\|^6_6 \leq\|S_{N, m}\|^4|E_{N, m}|\\
 \notag
 &+\sum_{k=1}^m  \sum_{\substack{a, b\in E_{N, k}\\ (a, b)=1}}\left|E_{\frac{N}{\max(a, b)},\ m-k}\right|
 \sum_{a=a_1'a_2'} \sum_{b=b_1'b_2'} \S\left(\frac{N}{a_i'}, \frac{N}{b_i'}, m-\Omega(a_i'), m-\Omega(b_i'): i=1, 2\right),
 \end{align}
where 
\[\S(N_i, N_i', m_i, m_i': i=1, 2)\ =\sum_{\substack{a_1a_2=b_1b_2\\ \Omega(a_i)=m_i, \Omega(b_i)=m_i' \\ a_i\le N_i, \ b_i\le N_i'}} 1.\]
\end{prop}
\begin{proof} The following identity is from \cite{MR3430378}
\begin{equation}\label{eq:2nd_power}
 |S_{N, m}|^2=|E_{N, m}|+\sum_{k=1}^m \  \sum_{\substack{a, b\in E_{N, k}\\ (a, b)=1}}\left|E_{\frac{N}{\max(a, b)}, m-k}\right|z(a)\overline{z(b)}.
\end{equation}
For our computation, we will use the fact that 
\begin{equation}\label{eq:orthogonality}
\int_{T^\infty}z(n)\d m_\infty(z)=\begin{cases}
                                     1 \quad &\text{ if } n=1,\\
                                     0 \quad &\text{ if } n\geq 2.
                                    \end{cases} 
\end{equation}
Write $|S_{N, m}|^4$ as follows
\begin{equation}
 \label{eq:4th_power_1}
 |S_{N, m}|^4 =\sum_{\substack{a_i, b_i  \in E_{N, m} \\ i=1, 2}}z(a_1)z(a_2)\overline{z(b_1)z(b_2)}
 = \|S_{N, m}\|^4_4+\sum_{\substack{a_1a_2\ne b_1b_2 \\ a_i, b_i  \in E_{N, m}}}z(a_1)z(a_2)\overline{z(b_1)z(b_2)}.
\end{equation}
In the above sum we factor $a_1, a_2, b_1, b_2$ as  $a_i=a_i'a_i''$ and $b_i=b_i'b_i''$ such that 
\[(a_1a_2, b_1b_2)=a_1''a_2''=b_1''b_2''.\]
Further we write $a=a_1'a_2',\ b=b_1'b_2'$. The above factorizations of $a_i$ and $b_i$ are not necessarily unique. So we may introduce some extra terms in the expression for $|S_{N, m}|^4$ and write
\begin{align}
\label{eq:4th_power}
 |S_{N, m}|^4 
 \preceq\|S_{N, m}\|^4 + \sum_{k=1}^{2m} \ \sum_{\substack{a, b\in E_{N^2, k}\\ (a, b)=1}} \ \sum_{\substack{a=a_1'a_2'\\ a_i'\le N\\ \Omega(a_i')\le m}}
 \ \sum_{\substack{b=b_1'b_2'\\ b_i'\le N\\ \Omega(b_i')\le m}} 
 \ \ \sum_{\substack{a_1''a_2''=b_1''b_2''\\ \Omega(a_i'')=m-\Omega(a_i')\\ \Omega(b_i'')=m-\Omega(b_i')\\ a_i''\le \frac{N}{a_i'}, \ b_i''\le \frac{N}{b_i'}}} z(a)\overline{z(b)}.
\end{align}
The symbol \lq$\preceq$\rq above means that we have some extra terms in the left hand side of the expression.

The expression for $6^{th}$ norm in (\ref{eq:identity_6thmoment}) follows by matching the terms in (\ref{eq:4th_power}) with its conjugate in (\ref{eq:2nd_power}) and then using (\ref{eq:orthogonality}).
\end{proof}

Later we will use the notation $k_i=\Omega(a_i')$ and $ k_i'=\Omega(b_i')$ for $i=1, 2$.
%
%
%
\section{Integers With Given Number of Prime Factors}\label{sec:lemm}
Now we will give upper bounds for some expression involving $|E_{N, m}|$, which will be used later in the proof of our theorem.

The function $|E_{N, m}|$ is of interest in number theory, and studied extensively in the literature starting from the prime number theorem. Following 
estimate for $|E_{N, m}|$ is due to Sathe\cite{MR0056625}.
\begin{lem}\label{lem:EN}
For $N\rightarrow\infty$ and $1\leq m\leq (2-\epsilon)\log\log N$ for $0<\epsilon<1$, we have
\[|E_{N, m}|\asymp\frac{N(\log\log N)^{m-1}}{(m-1)!\log N}.\]
\end{lem}

Later Balazard, Delange and Nicolas\cite{MR941613} generalized this result to an uniform range of $m$:

\begin{lem}\label{lem:gen_EN}
For $m\geq 1$ and $\frac{x}{2^m}\rightarrow\infty$, we have
\[|E_{N, m}|\asymp\frac{N}{2^m\left(\log\frac{N}{2^m}\right)}\sum_{j=0}^{m-1}\frac{\left(2\log\log\frac{N}{2^m}\right)^j }{j!}.\]
\end{lem}

We will use Lemma~\ref{lem:EN} and Lemma~\ref{lem:gen_EN} to prove the following results. 
Lemma~\ref{lem:Ebk/bb} has also appeared in \cite{MR3430378}.

\begin{lem}\label{lem:Ebk/bb}
For $k\leq \log\log N$ and $N, k\rightarrow\infty,$ we have
$$\sum_{b\in E_{N, k}}\frac{\left|E_{b, k}\right|}{b^2}\ll 2^{2k}.$$
\end{lem}
\begin{proof}
Using Lemma~\ref{lem:gen_EN}
\begin{align*}
 \sum_{b\in E_{N, k}}\frac{\left|E_{b, k}\right|}{b^2}&\ll\int_{2^k}^N\frac{1}{2^kx\log\frac{x}{2^k}}
 \sum_{0\leq j_1<k}\frac{\left(2\log\log\frac{x}{2^k}\right)^{j_1}}{j_1!}\d E_{x, k}\\
 &\ll\int_{2^k}^N\frac{1}{2^{2k}x\left(\log\frac{x}{2^k}\right)^2}\sum_{0\leq j_1, j_2<k}\frac{\left(2\log\log\frac{x}{2^k}\right)^{j_1+j_2}}{j_1!j_2!}\d x\\
 &\ll\frac{1}{2^{2k}}\sum_{0\leq j_1, j_2<k}\frac{2^{j_1+j_2}}{j_1!j_2!}\int_{2^k}^N\frac{\left(\log\log\frac{x}{2^k}\right)^{j_1+j_2}}{x\left(\log\frac{x}{2^k}\right)^2}\d x\\
 &\ll\frac{1}{2^{2k}}\sum_{0\leq l<2k-1}\sum_{ j_1+ j_2=l}2^{j_1+j_2}\frac{\Gamma(j_1+j_2+1)}{j_1!j_2!}\\
 &\ll\frac{1}{2^{2k}}\sum_{0\leq l<2k-1}2^{2l}\ll2^{2k}.
\end{align*}
In the above computation we have used the following formula for Gamma function:
\[\Gamma(n+1)=\int_0^\infty x^ne^{-x}\d x= n!.\]
\end{proof}



\begin{lem}\label{lem:sumENb}
Let $k, k'\leq \log\log N$ and $N\rightarrow\infty$. Then
 \[\sum_{\substack{b \in E_{N, k}\\ b>\sqrt{N}}}\left|E_{\frac{N}{b}, k'}\right|\ll \left|E_{N, k}\right|\frac{(\log\log N)^{k'}}{k'!}.\]
\end{lem}
\begin{proof}
Using Lemma~\ref{lem:gen_EN}, we simplify the above sum as follows:
\begin{align*}
 \sum_{\substack{b \in E_{N, k}\\ b>\sqrt{N}}}\left|E_{\frac{N}{b}, k'}\right| 
 &\ll \int_{\sqrt N}^{N/2^k}\frac{N}{x2^{k'}\left(\log\frac{N}{x2^{k'}}\right)}\sum_{j=0}^{k'-1}\frac{\left(2\log\log\frac{N}{x2^{k'}}\right)^j }{j!}\d |E_{x, k}|\\
 &\ll |E_{N, k}|\sum_{j=0}^{k'-1}\frac{1}{j!2^{k'-j}}\int_{\sqrt N}^{N/2^k}\frac{\left(\log\log\frac{N}{x2^{k'}}\right)^j }{x\log\frac{N}{x2^{k'}}}\d x \\
 &\ll |E_{N, k}|\sum_{j=0}^{k'-1}\frac{1}{j!2^{k'-j}}\int_{\log((k-k')\log2)}^{\log(\half\log N + (k-k')\log 2 )}y^j\d y\\
 &\ll |E_{N, k}|\sum_{j=0}^{k'-1}\frac{(\log\log N)^{j+1}}{(j+1)!2^{k'-j}}\ll \left|E_{N, k}\right|(\log\log N)^{k'}.
\end{align*}
\end{proof}

%

\section{Upper Bound For $\S$} \label{sec:S}
In this section we will obtain some estimates for upper bound of $\S$.
Recall
\[\S(N_i, N_i', m_i, m_i': i=1, 2)\ =\sum_{\substack{a_1a_2=b_1b_2\\ \Omega(a_i)=m_i, \Omega(b_i)=m_i' \\ a_i\le N_i, \ b_i\le N_i'}} 1.\]
\begin{lem}\label{lem:formula_S}
Let 
\begin{align*}
 & N=N_1'N_2'=\min(N_1N_2, N_1'N_2') \text{ and }\\
 & m=m_1+m_2=m_1'+m_2'.\\
\end{align*}
Then
\begin{align*}
 &\S(N_i, N_i', m_i, m_i': i=1, 2)\\
 &\ll\left\{
 \left|E_{N_1', m_1'}\right|\left|E_{N_2', m_2'}\right|\left(\left|E_{N_1, m_1}\right|\left|E_{\frac{N}{N_1}, m_2}\right|
 +\left|E_{\frac{N}{N_2}, m_1}\right|\left|E_{N_2, m_2}\right|\right)\right\}^\half.
\end{align*}
\end{lem}
\begin{proof}
By Cauchy-Schwarz inequality 
\begin{align*}
 &S(N_i, N_i', m_i, m_i': i=1, 2)
 = \int\limits_{\mathbb{T}^{\infty}}\left(\sum_{\substack{n\\ n=n_1n_2\\n_i\in E_{N_i, m_i}}}z(n)\right)
 \overline{\left(\sum_{\substack{n'\\ n'=n_1'n_2' \\n_i'\in E_{N_i', m_i'}}}z(n_i')\right)}\d\mu_\infty\\
 &\ll \left(\int\limits_{\mathbb{T}^{\infty}}\left|\sum_{\substack{n\leq N\\ n=n_1n_2\\n_i\in E_{N_i, m_i}}}z(n)\right|^2\d\mu_\infty\right)^\half
 \left(\int\limits_{\mathbb{T}^{\infty}}\left|\sum_{\substack{n'\leq N\\ n'=n_1'n_2' \\n_i'\in E_{N_i', m_i'}}}z(n_i')\right|^2\d\mu_\infty\right)^\half\\
 &\ll\left\{ \left|E_{N_1', m_1'}\right|\left|E_{N_2', m_2'}\right|\left(\left|E_{N_1, m_1}\right|\left|E_{\frac{N}{N_1}, m_2}\right|
 +\left|E_{\frac{N}{N_2}, m_1}\right|\left|E_{N_2, m_2}\right|\right)\right\}^\half.
\end{align*}
\end{proof}


\section{Proof of Theorem~\ref{thm:main}}\label{sec:proof}
Clearly $\|S_{N, m}\|_6^6\gg |E_{N, m}|^3$. So to prove Theorem~\ref{thm:main}, it is sufficient to show 
\[\|S_{N, m}\|_6^6\ll |E_{N, m}|^3.\]
For $m<\half\log\log N$, it is shown in \cite{MR3430378} that $\|S_{N, m}\|_4^4\ll |E_{N, m}|^2$. This reduce the problem to show that the sum in the left hand side of (\ref{eq:identity_6thmoment}) is bounded by
$|E_{N, m}|^3$. We divide this sum in the following 2-parts and estimate each of them separately:
\begin{align*}
 & \sum_{k=1}^m  \sum_{\substack{a, b\in E_{N, k}\\ (a, b)=1}}\left|E_{\frac{N}{\max(a, b)},\ m-k}\right|
 \sum_{a=a_1'a_2'} \sum_{b=b_1'b_2'} \S\left(\frac{N}{a_i'}, \frac{N}{b_i'}, m-\Omega(a_i'), m-\Omega(b_i'): i=1, 2\right) \\
 \leq&8\sum_{k=1}^m\sum_{b\in E_{N, k}}\left|E_{\frac{N}{b},\ m-k}\right|\sum_{\substack{ a\in E_{b, k}\\ (a, b)=1}} \ \sum_{\substack{a=a_1'a_2'\\ a_1'\leq a_2'}}
 \ \sum_{\substack{b=b_1'b_2'\\ b_1'\leq b_2'}}\S\left(\frac{N}{a_i'}, \frac{N}{b_i'}, m-\Omega(a_i'), m-\Omega(b_i'): i=1, 2\right)\\
 =&\sum_{k=1}^m \left(A_1(k)+A_2(k)\right),
\end{align*}
where
\begin{align*}
A_1(k)&= \sum_{\substack{ b\in E_{N, k}\\ b\leq \sqrt N}}\left|E_{\frac{N}{b},\ m-k}\right|\sum_{\substack{ a\in E_{b, k}\\ (a, b)=1}} \ \sum_{\substack{a=a_1'a_2'\\ a_1'\leq a_2'}}
 \ \sum_{\substack{b=b_1'b_2'\\ b_1'\leq b_2'}}\S\left(\frac{N}{a_i'}, \frac{N}{b_i'}, m-\Omega(a_i'), m-\Omega(b_i'): i=1, 2\right), \\
A_2(k)&= \sum_{\substack{ b\in E_{N, k}\\ b> \sqrt N}}\left|E_{\frac{N}{b},\ m-k}\right|\sum_{\substack{ a\in E_{b, k}\\ (a, b)=1 }}
\ \sum_{\substack{a=a_1'a_2'\\ a_1'\leq a_2'}} \ \sum_{\substack{b=b_1'b_2'\\ b_1'\leq b_2'}}
\S\left(\frac{N}{a_i'}, \frac{N}{b_i'}, m-\Omega(a_i'), m-\Omega(b_i'): i=1, 2\right).\\
\end{align*}

We also recall the notations  $k_i=\Omega(a_i')$ and $ k_i'=\Omega(b_i')$ for $i=1, 2$.
\subsection{Computation for $A_1$}
Using upper bounds for $\S$ and $|E_{N/b, m-k}|$ from Lemma~\ref{lem:formula_S} and Lemma~\ref{lem:EN} respectively, we simplify the expression for $A_1(k)$ as follows
\begin{align*}
 &A_1(k) \ll |E_{N, m-k}|\sum_{\substack{b\in E_{N, k}\\ b\leq \sqrt N}}\frac{1}{b} \sum_{\substack{ a\in E_{b, k}\\ (a, b)=1}} \ \sum_{\substack{a=a_1'a_2'\\ a_1'\leq a_2'}}
 \ \sum_{\substack{b=b_1'b_2'\\ b_1'\leq b_2'}}\\
 &\times\left\{|E_{N/b_1', m-k_1'}||E_{N/b_2', m-k_2'}|(|E_{N/a_1, m-k_1}||E_{Na_1/b, m-k_2}|+|E_{Na_2/b, m-k_1}||E_{N/a_2, m-k_2}|)\right\}^{\half}\\
 &\ll 2^{2k}|E_{N, m-k}|\sum_{\substack{b\in E_{N, k}\\ b\leq \sqrt N}}\frac{|E_{b, k}|}{b^2} \sum_{\substack{k=k_1+k_2\\ ~=k_1'+k_2'}}
 \left(\left|E_{N, m-k_1}\right|\left|E_{N, m-k_1'}\right|\left|E_{N, m-k_2}\right|\left|E_{N, m-k_2'}\right|\right)^{\half}.
\end{align*}
By Lemma~\ref{lem:EN}
\begin{align*}
&|E_{N/b_1', m-k_1'}||E_{N/b_2', m-k_2'}|(|E_{N/a_1, m-k_1}||E_{Na_1/b, m-k_2}|+|E_{Na_2/b, m-k_1}||E_{N/a_2, m-k_2}|\\ 
&\ll \frac{1}{b^2}\left|E_{N, m-k_1}\right|\left|E_{N, m-k_1'}\right|\left|E_{N, m-k_2}\right|\left|E_{N, m-k_2'}\right|,
\end{align*}
and trivially $\sum_{\substack{a=a_1'a_2'\\ a_1'\leq a_2'}}\sum_{\substack{b=b_1'b_2'\\ b_1'\leq b_2'}}1\ll 2^{2k}$. These justifies the above bound for $A_1(k)$.
By Lemma~\ref{lem:Ebk/bb} $\sum_{\substack{b\in E_{N, k}\\ b\leq \sqrt N}}\frac{|E_{b, k}|}{b^2}\ll 2^{2k}$. This simplifies the bound for $A_1(k)$ to
\[A_1(k) \ll 4^{2k}|E_{N, m-k}|\sum_{\substack{k=k_1+k_2\\ ~=k_1'+k_2'}}\left(\left|E_{N, m-k_1}\right|\left|E_{N, m-k_1'}\right|\left|E_{N, m-k_2}\right|\left|E_{N, m-k_2'}\right|\right)^{\half}.\]
Now we divide $A_1(k)$ by $|E_{N, m}|^3$, use Lemma~\ref{lem:EN}, and sum over $k$ to get
\begin{align*}
 &\sum_{k=1}^m\frac{A_1(k)}{\left|E_{N, m}\right|^3}
 \ll \sum_{k=1}^{m}4^{2k}\frac{|E_{N, m-k}|}{|E_{N, m}|}\sum_{\substack{k=k_1+k_2\\ ~=k_1'+k_2'}}
 \frac{\left(\left|E_{N, m-k_1}\right|\left|E_{N, m-k_1'}\right|\left|E_{N, m-k_2}\right|\left|E_{N, m-k_2'}\right|\right)^{\half}}{\left|E_{N, m}\right|^2}\\
 &\ll \sum_{k=1}^{m}k\left(\frac{4m}{\log\log N}\right)^{2k}. 
\end{align*}
So 
\[\sum_{k=1}^m\frac{A_1(k)}{\left|E_{N, m}\right|^3}\ll 1, \text{ when } m<c\log\log N, \ c<1/4.\] 

We believe that $A_1(k)$ and $A_2(k)$ are of similar size, which suggests us to conjecture that $c(3)$ (as defined in Section~\ref{sec:intro}) is $\frac{1}{4}$. 
\subsection{Computation for $A_2$}
By Lemma~\ref{lem:formula_S}  
\begin{align*}
&\S\left(\frac{N}{a_i'}, \frac{N}{b_i'}, m-\Omega(a_i'), m-\Omega(b_i'): i=1, 2\right)\\
&\ll \left\{\left|E_{\frac{N}{b_1'}, m-k_1'}\right|\left|E_{\frac{N}{b_2'}, m-k_2'}\right|\left(\left|E_{\frac{N}{a_1'}, m-k_1}\right|\left|E_{\frac{a_1'N}{b}, m-k_2}\right|+\left|E_{\frac{a_2'N}{b}, m-k_1}\right|\left|E_{\frac{N}{a_2'}, m-k_2}\right|\right)\right\}^\half.
\end{align*}
We only consider that part of the sum that involves $a_1'$, and the computation for the sum involving $a_2'$ is similar. Note that as $a_1'\leq a_2', \ b_1'\leq b_2'$ and $a_1'a_2', b_1'b_2' \leq N$. So we have 
$a_1', b_1' \leq \sqrt N$. This and Lemma~\ref{lem:EN} implies
\begin{align*}
&\frac{b}{N}\left\{\left|E_{\frac{N}{b_1'}, m-k_1'}\right|\left|E_{\frac{N}{b_2'}, m-k_2'}\right|\left|E_{\frac{N}{a_1'}, m-k_1}\right|\left|E_{\frac{a_1'N}{b}, m-k_2}\right|\right\}^\half\\
&\ll|E_{N, m-k_1'}|^\half|E_{N, m-k_1}|^\half\frac{b}{N\sqrt{a_1'b_1'}} \left|E_{\frac{N}{b_2'}, m-k_2'}\right|^\half\left|E_{\frac{a_1'N}{b}, m-k_2}\right|^\half\\
&\ll|E_{N, m-k_1'}|^\half|E_{N, m-k_1}|^\half.
\end{align*}
We can do similar computation for remaining part of $\S$ and show
\[\S\left(\frac{N}{a_i'}, \frac{N}{b_i'}, m-\Omega(a_i'), m-\Omega(b_i'): i=1, 2\right)\ll\frac{N}{b}|E_{N, m-k_1'}|^\half\left(|E_{N, m-k_1}|^\half+|E_{N, m-k_2}|^\half\right).\]
So
\begin{align*}
&A_2(k)\\
&= \sum_{\substack{ b\in E_{N, k}\\ b> \sqrt N}}\left|E_{\frac{N}{b}, m-k}\right| \sum_{\substack{ a\in E_{b, k}\\ (a, b)=1 }} \ \ 
\ \sum_{\substack{a=a_1'a_2'\\ a_1'\leq a_2'}} \ \sum_{\substack{b=b_1'b_2'\\ b_1'\leq b_2'}}\S\left(\frac{N}{a_i'}, \frac{N}{b_i'}, m-\Omega(a_i'), m-\Omega(b_i'): i=1, 2\right)\\
&\ll\sum_{\substack{ b\in E_{N, k}\\ b> \sqrt N}}\left|E_{\frac{N}{b}, m-k}\right| \sum_{\substack{ a\in E_{b, k}\\ (a, b)=1 }} \frac{N}{b} 
\ \sum_{\substack{a=a_1'a_2'\\ a_1'\leq a_2'}} \ \sum_{\substack{b=b_1'b_2'\\ b_1'\leq b_2'}}|E_{N, m-k_1'}|^\half\left(|E_{N, m-k_1}|^\half+|E_{N, m-k_2}|^\half\right).
\end{align*}
As $\sum_{\substack{a=a_1'a_2'\\ a_1'\leq a_2'}} \sum_{\substack{b=b_1'b_2'\\ b_1'\leq b_2'}}1\ll \sum_{k_1'=0}^{k}\sum_{k_1=0}^{k}{k \choose k_1'}{k \choose k_1}$, we have
\begin{align*}
&A_2(k)\\
&\ll\sum_{k_1'=0}^{k}\sum_{k_1=0}^{k}{k \choose k_1'}{k \choose k_1}|E_{N, m-k_1'}|^\half\left(|E_{N, m-k_1}|^\half+|E_{N, m-k_2}|^\half\right)\sum_{\substack{ b\in E_{N, k}\\ b> \sqrt N}}\frac{N}{b}|E_{b, k}|\left|E_{\frac{N}{b}, m-k}\right|\\
&\ll\sum_{k_1'=0}^{k}\sum_{k_1=0}^{k}{k \choose k_1'}{k \choose k_1}|E_{N, m-k_1'}|^\half\left(|E_{N, m-k_1}|^\half+|E_{N, m-k_2}|^\half\right)|E_{N, k}|\sum_{\substack{ b\in E_{N, k}\\ b> \sqrt N}}\left|E_{\frac{N}{b}, m-k}\right|.
\end{align*}
As $k_1$ and $k_2$ are symmetric, we may drop $|E_{N, m-k_2}|^\half$ from the above inequality. Further by applying Lemma~\ref{lem:sumENb} to $\sum_{\substack{ b\in E_{N, k}\\ b> \sqrt N}}|E_{\frac{N}{b}, m-k}|$, we get
\begin{align*}
A_2(k)
\ll \sum_{k_1'=0}^{k}\sum_{k_1=0}^{k}{k \choose k_1'}{k \choose k_1}|E_{N, m-k_1'}|^\half|E_{N, m-k_1}|^\half |E_{N, k}|^2(\log\log N)^{m-k}.
\end{align*}
Now divide $A_2(k)$ by $|E_{N, m}|^3$, apply Lemma~\ref{lem:EN}, and sum over $k$ to have
\begin{align*}
 \sum_{k=1}^{m}\frac{A_2(k)}{|E_{N, m}|^3}&\ll \sum_{k=1}^{m}\sum_{k_1'=0}^{k}\sum_{k_1=0}^{k}{k \choose k_1'}{k \choose k_1}\left(\frac{m}{\log\log N}\right)^{\frac{k_1'+k_1}{2}}
 \left(\frac{m^2}{\log\log N}\right)^{m-k}\\
 &\ll \sum_{k=1}^{m}\sum_{k_1'=0}^{k}\sum_{k_1=0}^{k}\frac{1}{k_1!k_1'!}\left(\frac{k^2m}{\log\log N}\right)^{\frac{k_1'+k_1}{2}}\left(\frac{m^2}{\log\log N}\right)^{m-k}.
 \end{align*}
When $m\leq C (\log\log N)^{\frac{1}{3}}$ for some constant $C>0$,
\begin{align*}
 \sum_{k=1}^{m}\frac{A_2(k)}{|E_{N, m}|^3}\ll\sum_{k=1}^{m}\sum_{k_1'=0}^{k}\sum_{k_1=0}^{k}\frac{1}{k_1!k_1'!} \frac{C^{\frac{3(k_1'+k_1)}{2}+ 2(m-k)}}{(\log\log N)^{\frac{m-k}{3}}}\ll_C 1.
\end{align*}
This completes the proof of our theorem.
 
 \subsection*{Acknowledgement}
We would like to thank Andriy Bondarenko and Kristian Seip for several insightful discussions.

\bibliographystyle{abbrv}
\bibliography{ref_6thmom}

\end{document}